\documentclass[a4paper,leqno,12pt,twoside]{amsart}
\usepackage{amsmath,amsfonts,amssymb,amsthm,amscd}
\usepackage[utf8]{inputenc}
\usepackage[T1]{fontenc}
\usepackage[english]{babel}
\usepackage[colorlinks=true,citecolor=blue, urlcolor=blue, linkcolor=blue,pagebackref]{hyperref}
\usepackage[top=1.2in,bottom=1.2in,left=1.in,right=1.in]{geometry} 
\usepackage{graphicx}
\usepackage{paralist}
\usepackage{tabto}
\usepackage{standalone}
\usepackage{tikz}
\usetikzlibrary{matrix}
\usetikzlibrary{arrows}
\usepackage{xfrac}
\usepackage{amsthm, amsmath, amssymb,latexsym}
\usepackage{tikz-cd}
\usepackage[all]{xy} 

\usepackage{enumerate}
\usepackage{xcolor}
\usepackage{aliascnt}
\usepackage{cleveref}

\newtheorem{theorem}{Theorem}[section]
\newtheorem{lemma}[theorem]{Lemma}
\newtheorem{corollary}[theorem]{Corollary}
\newtheorem{proposition}[theorem]{Proposition}
\newtheorem{remark}[theorem]{Remark}
\newtheorem{definition}[theorem]{Definition}

\def\ra{\rightarrow}

\def\Mbar{\overline{\mathcal{M}}}
\def\M{\mathcal{M}}
\def\Bbar{\overline{\mathcal{B}}}
\def\B{\mathcal{B}}
\def\P{\mathcal{P}}
\def\Pbar{\overline{\mathcal{P}}}

\title[Non-tautological cycles on moduli spaces of curves]{Non-tautological cycles on moduli spaces of smooth pointed curves}
\author{Dario Faro}

\address{Dario Faro  \\ Universit\`a degli Studi di Pavia  \\ Dipartimento di Matematica \\ Via Ferrata 5  \\ 27100 Pavia, Italy  }
\email{dario.faro@unipv.it}

\author{Carolina Tamborini}

\address{Carolina Tamborini  \\ Universit\"at Duisburg-Essen \\ Fakult\"at für Mathematik \\ Thea-Leymann-Str. 9 \\ 45127 Essen }
 \email{carolina.tamborini@uni-due.de}

\begin{document}
	\begin{abstract}
In recent work by Arena, Canning,  Clader,  Haburcak, Li, Mok, and Tamborini it was proven that for infinitely many values of $g$ and $n$, there exist non-tautological algebraic cohomology classes on the moduli space $\M_{g,n}$ of smooth genus $g$, $n$-pointed curves. 
Here we show how a  generalization of their technique allows to cover most of the remaining cases, proving the existence of non-tautological algebraic cohomology classes on the moduli space $\M_{g,n}$ for all but finitely many values of $g$ and $n$.

	\end{abstract}

	\thanks{
	The authors are members of GNSAGA (INdAM). D. Faro  is  partially supported by PRIN project {\em Moduli spaces and special varieties} (2022), while C. Tamborini is partially supported by DFG-Research Training Group 2553 “Symmetries and classifying spaces: analytic, arithmetic, and derived”}
	
	\maketitle

	\section{Introduction}
	
	Let $\M_{g,n}$ and $\Mbar_{g,n}$ be the moduli spaces of genus $g$, $n-$pointed smooth and stable curves, respectively. The study of the cohomology of these moduli spaces is a very interesting and difficult one. While some results are known for very small $g$ and $n$, making general statements on the full cohomology ring of $\Mbar_{g,n}$ and $\M_{g,n}$ is very hard. A key strategy to approach this problem is to analyse their \emph{tautological rings}. These are natural subrings $RH^*(\Mbar_{g,n}) \subseteq H^*(\Mbar_{g,n})$, resp.  $RH^*(\M_{g,n}) \subseteq H^*(\M_{g,n})$, that are easier to work with but still contain a large number of interesting algebraic cohomology classes. In this setting, a fundamental problem becomes understanding the relation between the cohomology ring $H^*$ and its tautological subring $RH^*$.    For  $\Mbar_{0,n}$, it has been proved in \cite{Keel} that  all cohomology classes are tautological, and in \cite{Pet14} it is shown  that all 
	even-degree cohomology classes on $\Mbar_{1,n}$  are tautological.
 As the tautological ring $RH^*(\M_{g,n}) \subseteq H^*(\M_{g,n})$ is defined as the restriction of  $RH^*(\Mbar_{g,n}) \subseteq H^*(\Mbar_{g,n})$ to the interior, these results also imply that all algebraic cohomology classes on $\M_{0,n}$ and $\M_{1,n}$ are tautological.
 We refer to \cite{FP} for a description of the main properties of the tautological ring and of the methods for detecting non-tautological classes.
 
	This note deals with the construction of explicit examples of algebraic non-tautological cohomology classes in $H^*(\Mbar_{g,n})$ and $H^*(\M_{g,n})$, generalizing work of \cite{GP}, \cite{vZ} and \cite{AGNES}, as we now explain. Consider the double cover cycle
	 \begin{gather*}
		\B_{g\ra h, n, 2m}:=\{(C; x_1,\dots, x_n, p_1, q_1,\dots, p_m, q_m)\in \M_{g,2m+n} \mid 
		\exists f: C\overset{2:1}{\longrightarrow}D,\ g(D)=h,\\ x_i\text{ is a ramification point of $f$ for }i=1,\dots,n,\ \text{and}\  f(p_i)=f(q_i)\ i=1,\dots, m \},
	\end{gather*}
    and denote by $\Bbar_{g\ra h, n, 2m}$ its closure in $\Mbar_{g,2m+n}$.
    The first example of an explicit algebraic cycle which is not tautological is due to Graber and Pandharipande \cite{GP}: they prove that
\begin{gather*}
    [\Bbar_{2\ra 1, 0, 20}]\in H^*(\Mbar_{2,20})\quad \text{and} \quad  [\B_{2\ra 1,0, 20}]\in H^*(\M_{2,20})
\end{gather*}
   are non-tautological. Subsequentely, their method and result has been extensively extended. First by Van Zelm \cite{vZ}, who proved that $[\Bbar_{g\ra 1, n, 2m}]\in H^*(\Mbar_{g,2m+n})$ is non-tautological for all $g+m\geq 12$, $g\geq 2$, and $0\leq n \leq 2g-2$. This establishes the existence of non-tautological algebraic classes in $H^*(\Mbar_{g,n})$ for all but finitely many $g$ and $n$. For finitely many cases, van Zelm also proves that the cycles $[\Bbar_{g\ra 1, n, 2m}]$ restrict to non-tautological algebraic cycles on the interior. The precise statement is that $[\B_{g\ra 1,0, 2m}]\in H^*(\M_{g,2m})$ is non-tautological for $g+m=12$ and $g\geq 2$.

    A significant progress in the approach of the problem for the interior is the one in \cite{AGNES}, where Arena, Canning,  Clader,  Haburcak, Li, Mok, and Tamborini prove that 
    $$[\B_{g\ra 1,0, 2m}]\in H^*(\M_{g,2m})$$ is non-tautological also for $g+m$ even, $g+m\geq 16$, and $g\geq 2$, and that
    $$[\B_{g\ra 2,0, 2m}]\in H^*(\M_{g,2m})$$ is non-tautological for $g+m$ odd, $g+m\geq 17$, and $g\geq 4$.
    Moreover, they also show the existence of non-tautological algebraic classes for $H^*(\M_{g,1})$ when $g=12$ or $g\geq 16$. 
    This provides the existence of non-tautological algebraic cohomology classes in $H^*(\M_{g,n})$, for infinitely many values of $g$ and $n$, as long as $n$ is even or $n=1$. In particular note that, for $g\geq 4$, it provides non-tautological algebraic cycles on $\M_{g, 2n}$ for all but finitely many $g$ and $n$.\\

In this work we show how a slight generalization of the setting and the proofs allows to extend the results in \cite{AGNES} to cover most of the remaining cases, proving the existence of non-tautological algebraic cohomology classes on the moduli space $\M_{g,n}$ for all but finitely many values of $g$ and $n$. More precisely, our proof gives the  existence of non-tautological algebraic classes in $\M_{g,2n+1}$, for all but finitely many values of $g$ and $n$, and covers some remaining cases when the number of marked points is even. Together with the previous work, this gives. 
	\begin{theorem}
			\label{teoprincipale}
		  There exists a non-tautological algebraic cycle on $\M_{g,n}$ for any $n\geq 0$ such either $2\leq g\leq 12$ and $2g+n\geq 24$, or $2g+n\geq 32$. In particular, for such $g$ and $n$ we have
		 	$RH^{2*}(\M_{g,n})\neq H^{2*}(\M_{g,n})$.
	\end{theorem}
	
	The new contributions of this work in Theorem \ref{teoprincipale} are the construction of non-tautological algebraic cohomology classes for
	\begin{enumerate}
		\item $\M_{g,n}$ for any $g$ and $n$ as in Theorem \ref{teoprincipale} when $n$ is odd and $n\neq 1$;
		\item $\M_{g,n}$ for $2\leq g\leq 12$ and $24 < 2g+n < 32$, when $n$ is even;
        \item $\M_{g,2n}$ for $g=2,3$, $g+n$ odd $g+n\geq 17$.
	\end{enumerate}

 The algebraic cycles that we prove being non-tautological are the pullback of the double cover cycles considered in \cite{AGNES} via morphisms forgetting a certain number of marked points, and the proof is an adaptation of the one in \cite{AGNES} to this generalized setting. \\

\textbf{Acknowledgments} We thank Samir Canning for an email  related to this work.

	\section{Preliminaries on the tautological ring of $\Mbar_{g,n}$}

	In this section we recall some preliminaries on non-tautological classes in the cohomology $H^*(\Mbar_{g,n})$, and collect some facts and remarks that we will use in the following.\\
	
	The system of tautological subrings $$RH^*(\Mbar_{g,n}) \subseteq H^*(\Mbar_{g,n})$$ is defined as the smallest system of $\mathbb{Q}$-subalgebras  that 
	
	\begin{enumerate}
		\item contains the fundamental classes $[\Mbar_{g,n}]$;
		\item is closed under pushforwards via all the maps forgetting a marking
		\begin{gather*}
			\Mbar_{g,n} \ra \Mbar_{g,n-1};
		\end{gather*}
		\item is closed under pushforwards  via all the gluing maps
			\begin{gather*}
			 \Mbar_{g_1,n_1+1}\times \Mbar_{g_2,n_2+1} \ra  \Mbar_{g_1+g_2,n_1+n_2}, \\
			 \Mbar_{g-1,n+2}\ra \Mbar_{g,n}.
		\end{gather*}
		
	\end{enumerate} 
	The tautological subring $RH^*(\M_{g,n}) \subseteq H^*(\M_{g,n})$ is defined as the image of $RH^*(\Mbar_{g,n})$ under the restriction map. \\

	Tautological classes are by definition algebraic, hence a straightforward observation is:

	\begin{itemize}
		\item Tautological classes have only contributions in $H^{2k}(\Mbar_{g,n})$.
		\item More precisely, tautological classes have only contributions in $H^{k,k}(\Mbar_{g,n})$.
	\end{itemize}
	
	In particular, the existence of holomorphic forms on some $\Mbar_{g,n}$, gives examples of non-tautological classes. In this note, a key ingredient will be the following result
	
	\begin{proposition}[\cite{GP}, Prop. 1]\label{GrPa}
		Let $i: \Mbar_{g_1,n_1+1}\times \Mbar_{g_2,n_2+1} \ra  \Mbar_{g_1+g_2,n_1+n_2}$ be a gluing map. If $\gamma\in RH^*(\Mbar_{g_1+g_2,n_1+n_2})$, then $i^*(\gamma)$ has a tautological K\"unneth decomposition, i.e.: 
		\begin{gather*}
			i^*(\gamma)\in RH^*(\Mbar_{g_1,n_1+1})\otimes RH^*( \Mbar_{g_2,n_2+1}).
		\end{gather*}
	\end{proposition}
	
The proposition implies that the pullback of a tautological class via a gluing morphism has only contributions of type $H^{p,p}(\Mbar_{g_1,n_1+1}) \otimes H^{q,q}( \Mbar_{g_2,n_2+1})$. As in \cite{AGNES}, and similarly to \cite{GP} and \cite{vZ}, we will show that some algebraic classes are non-tautological, by showing that their pullback via some gluing morphism has a non-zero contribution of type $H^{p,0}\otimes H^{0,p}$. We will need the following:
	
	\begin{proposition}\label{nonvanishing} The following non-vanishing results hold:
		\begin{enumerate}
			\item 	$H^{k,0}(\Mbar_{1,k})\neq 0$ for all $k\geq 11$ odd and $k\neq 13$ 
			\item	$H^{k+3,0}(\Mbar_{2,k})\neq 0$ for all $k\geq 14$ even. 
		\end{enumerate}
	\end{proposition}
	 For the first statement see e.g. \cite[Proposition 2,2]{CLP}, the second is \cite[Proposition 2.6]{AGNES}. From this immediately we have:

	\begin{lemma}\label{lemma1} The following non-vanishing results hold:
		\begin{enumerate}
			\item 	$H^{k,0}(\Mbar_{1,k+n})\neq 0$ for all $k\geq 11$ odd, $k\neq 13$, and $n\geq 0$;
			\item	$H^{k+3,0}(\Mbar_{2,k+n})\neq 0$ for all $k\geq 14$ even and $n\geq 0$.
		\end{enumerate}
	\end{lemma}
	\begin{proof}
		For $n=0$ this is the previous proposition. The non-vanishing for $n>0$ is obtained from $n=0$ by pulling back via the forgetful morphisms $\Mbar_{1,k+n}\ra  \Mbar_{1,k}$ resp. $\Mbar_{2,k+n}\ra \Mbar_{2,k}$.
	\end{proof}

    More precisely, we will use the non-vanishing results above together with the two following Lemmas. 
		
	\begin{lemma}\label{lemma2} If $H^{k,0}(\Mbar_{g,n})\neq 0 $, then there is a non-tautological term in the K\"unneth decomposition of $[\Delta]$, where $\Delta\subset \Mbar_{g,n} \times \Mbar_{g,n}$ denotes the diagonal.
	\end{lemma}
	\begin{proof}
		The Lemma is an immediate consequence of the K\"unneth decomposition of $[\Delta]$. Indeed, if $\{e_i\}_{i\in I}$ is a basis for $H^*(\Mbar_{g,n})$, then $[\Delta]=\sum_{i\in I} (-1)^{\deg e_i} e_i\otimes \hat{e_i}$, where $\{\hat{e_i}\}$ is the dual basis.
	\end{proof}
    \begin{lemma}\label{lemma3}
    	A class $0\neq e\in H^{k,0}(\Mbar_{g,n})$ restricts to a non-zero class in $H^k(\M_{g,n})$.
    \end{lemma}
    For the proof, see Fact 2.3 in \cite{AGNES}.
	  
	\section{Double cover cycles and their pullbacks}

	In this section we briefly provide the necessary definitions and background about the double cover cycles that will be object of this paper. Indeed, they will yield, in cohomology, the desired non-tautological classes. \\
	
	We start by borrowing the following definition of the moduli stack of admissible double covers from \cite{AGNES} and references therein.
	\begin{definition} We denote by $\overline{Adm}(g,h)_{r+2m}$ the stack of admissible double covers, parametrizing tuples
		\[
		(f:C \rightarrow D; x_1, \ldots, x_r; y_1, y_2, \dots, y_{2m-1}, y_{2m})
		\]
		where
		\begin{enumerate}
			\item $f:C\to D$ is a double cover of connected nodal curves of arithmetic genus $g$ and $h$, respectively;
			\item $x_1, \ldots, x_r \in C$ are precisely the smooth ramification points of $f$;
			\item $y_1,\dots, y_{2m}\in C$ are such that $f(y_{2i-1})=f(y_{2i})$;
			\item the image under $f$ of each node of $C$ is a node of $D$;
			\item the pointed curves $\big(C; (x_i)_{i=1}^r, (y_i)_{i=1}^{2m}\big)$ and $\big(D; (f(x_i))_{i=1}^r, f(y_{2i-1})_{i=1}^m\big)$ are stable. 
		\end{enumerate}	
		We call $\Bbar_{g\ra h, 2m}$ the image of $\overline{Adm}(g,h)_{r+2m}$ under the morphism 
		$$\overline{Adm}(g,h)_{r+2m}\ra \Mbar_{g,2m}$$ 
		forgetting the all the smooth ramification points, i.e., sending an admissible cover to the stabilization of $\big(C; (y_i)_{i=1}^{2m}\big)$. Finally, we denote by  $\B_{g\ra h, 2m}$ the restriction of $\Bbar_{g\ra h, 2m}$ to $\M_{g,2m}$.
		
	\end{definition}
	Note that $\B_{g\ra h, 2m}$ (resp.  $\Bbar_{g\ra h, 2m}$) writes, in the notations of \cite{AGNES}, as  $\B_{g\ra h, 0, 2m}$ (resp.  $\Bbar_{g\ra h, 0, 2m}$), i.e., here we are dropping the $0$ from the notation. 
	In \cite{AGNES} it is shown that the classes of the cycles $\Bbar_{g\ra h, 2m}$ and $\B_{g\ra h, 2m}$ in $H^*(\Mbar_{g,2m})$ (resp. $H^*(\M_{g,2m})$) are, under suitable numerical conditions, examples of non-tautological algebraic cohomology classes. In this work we will make use of the following modification of these cycles.

	\begin{definition}
		Consider the forgetful morphism
		$
		{p}_n: \M_{g,2m+n}\ra \M_{g,2m}
		$
		and set 
		\begin{gather*}	
			\P_{g\ra h, 2m, n}:=\B_{g\ra h, 2m} \times_{\M_{g,2m}} \M_{g,2m+n}.
		\end{gather*}
		Similarly let $\overline{p}_n: \Mbar_{g,2m+n}\ra \Mbar_{g,2m}$ and 
		\begin{gather*}
			\Pbar_{g\ra h, 2m, n}:=\Bbar_{g\ra h, 2m} \times_{\Mbar_{g,2m}} \Mbar_{g,2m+n}.
		\end{gather*}	
		
	\end{definition}
	
	Observe that if $\iota: \M_{g,2m+n} \rightarrow \Mbar_{g,2m+n}$ is the inclusion,  then 
	\begin{equation}
		\label{eq2}
		\iota^*[\Pbar_{g\ra h, 2m, n}]=[\P_{g\ra h, 2m, n}]. 
	\end{equation}
	
	As explained for example in \cite{AGNES}, one can easily compute that $\dim \Bbar_{g\ra h, 2m}=2g+m-h-1$. Hence $\dim \Pbar_{g\ra h, 2m, n}=2g+m-h-1+n$, and we get:
	\begin{gather*}
		\left[\Pbar_{g\to h, 2m, n }\right]\in H^{2(g+h+m-2)}\left(\Mbar_{g,2m+n}\right), \quad \left[ \P_{g\to h, 2m, n}\right] \in H^{2(g+h+m-2)} \left( \M_{g, 2m+n}\right).
	\end{gather*}
	Moreover by definition we have 
	\begin{gather}\label{classecoh}
		\left[\Pbar_{g\to h, 2m, n }\right]=(\overline{p}_n)^*	\left[\Bbar_{g\ra h, 2m}\right].
	\end{gather}
	{Note that, differently to  $\Bbar_{g\to h, 2m }$, $\Pbar_{g\to h, 2m, n }$ is not the image in $\Mbar_{g,2m+n}$ of a stack of admissible covers.}

	\section{Proof of the theorem}
	In this section we prove Theorem \ref{teoprincipale}. In the following, the statements and the proofs generalize those in \cite{AGNES},  which correspond to the case $n=0$.\\
	
	For $g, m, h, k, n \geq 0$ satisfying $g\geq 2h$ and $g+m=k+2h-1$, let
	\begin{gather*}
		i_n: \Mbar_{h,k} \times \Mbar_{h,k+n} \ra \Mbar_{g,2m+n}
	\end{gather*}
	be the gluing map defined by
	\[i_n\big( (C_1; x_1, \ldots, x_k), \; (C_2; y_1, \ldots, y_{k+n}) \big) = (C; x_{k-m+1}, y_{k-m+1}, \ldots, x_k, y_k, y_{k+1}, \ldots, y_{k+n}),\]
	where $C$ is the genus $g$ curve obtained by pairwise gluing the first $k-m$ marked points of $C_1$ to the first $k-m$ marked points of $C_2$. Of course, when $h=0$ we require that $k\geq 3$ and, when $h=1$ that $k\geq 1$.
	We denote by $j_n$ the composition
	\begin{gather}\label{eq:defj}
		j_n:\M_{h,k}\times \M_{h,k+n}\hookrightarrow \Mbar_{h,k}\times \Mbar_{h,k+n}\xrightarrow{i}\Mbar_{g,2m+n}.
	\end{gather}
	
	Finally, we denote by $\Delta_o$ the diagonal in $\M_{h,k}\times \M_{h,k}$, and let $\Delta_o^n:=\Delta_o \times_{(\M_{h,k}\times \M_{h,k})} \M_{h,k}\times \M_{h,k+n}$. Similarly, we denote by $\Delta$ the diagonal in $\Mbar_{h,k}\times \Mbar_{h,k}$, and $\Delta^n:=\Delta \times_{(\Mbar_{h,k}\times \Mbar_{h,k})} \Mbar_{h,k}\times \Mbar_{h,k+n}$. Observe that if $\alpha: \M_{h,k}\times \M_{h,k+n} \rightarrow \Mbar_{h,k}\times \Mbar_{h,k+n}$ is the inclusion, we have 
	\begin{equation}
		\label{eq3}
		\alpha^*[\Delta^n]=[\Delta^n_o].
	\end{equation}
	
	\begin{lemma}\label{diagonale} Let $g\geq 2h$ and $g+m=k+2h-1$. If $m\geq 1$ or $g\geq 4h+2$, then
		\begin{gather*}
			(j_n)^*[\Pbar_{g\ra h, 2m, n}]=\alpha [\Delta_o^n]
		\end{gather*}
		where $\alpha\in \mathbb{Q}_{>0}$.
	\end{lemma}
	\begin{proof}
		Consider the following commutative diagram

		\begin{equation*}\label{commutative diagram: gluing maps}
			\begin{tikzcd}
				{\M_{h,k}\times \M_{h,k+n}} \arrow[d, "id\times\pi_n"'] \arrow[r, "j_n"] & \Mbar_{g,2m+n} \arrow[d, "p_n"'] \\
				{\M_{h,k}\times \M_{h,k}} \arrow[r, "j"']        & {\Mbar_{g,2m}},                         
			\end{tikzcd}
		\end{equation*}
	
		where $\pi_n: \M_{h,k+n}\ra   \M_{h,k}$  is  the morphism forgetting the last $n$ points. By \cite[Proposition 1.5]{AGNES} we have $j^*[\Bbar_{g\ra h, 2m}]=\alpha [\Delta_o]$. Using this, together with \eqref{classecoh}, we get $(j_n)^*[\Pbar_{g\ra h, 2m, n}]=(j_n)^*(p_n)^*[\Bbar_{g\ra h, 2m}]=(p_n\circ j_n)^*[\Bbar_{g\ra h, 2m}]= (j\circ (id\times\pi_n))^*[\Bbar_{g\ra h, 2m}]=(id\times\pi_n)^*j^*[\Bbar_{g\ra h, 2m}]=(id\times\pi_n)^*(\alpha [\Delta_o])=\alpha [\Delta_o^n]$.
	\end{proof}

	\begin{theorem}\label{thm:main}
		Let $g, m, h, k\geq 0$ be such that $g\geq 2h$ and $g+m=k+2h-1$. Suppose that  $H^{3h-3+k,0}(\Mbar_{h,k})\neq 0$. If $m\geq 1$ or $g\geq 4h+2$ then 
		\begin{gather*}
			[\Pbar_{g\ra h, 2m, n}]\in H^{6h-6+2k} (\Mbar_{g,2m+n}) 
		\end{gather*}
		and \begin{gather*}
			[\P_{g\ra h, 2m, n}]\in H^{6h-6+2k} (\M_{g,2m+n})
		\end{gather*}
		are non-tautological for all $n\geq 0$.
	\end{theorem}
	\begin{remark}
	    For $n=0$ this is \cite[Theorem C]{AGNES}.
	\end{remark}
	\begin{proof}
		First observe that, since the tautological ring $RH^*(\M_{g,2m+n})\subset H^*(\M_{g,2m+n})$ is defined as the image of $RH^*(\Mbar_{g,2m+n})\subset H^*(\Mbar_{g,2m+n})$ under the restriction map, it suffices to show that $[\P_{g\ra h, 2m, n}]\in H^{6h-6+2k} (\M_{g,2m+n})$ is non-tautological for all $n\geq 0$.
		By excision and \eqref{eq2} we can write
		\begin{gather}\label{equality}
			[\Pbar_{g\ra h, 2m, n}]= [\P_{g\ra h, 2m, n}]+ B_n,
		\end{gather}
		where $B_n$ is an algebraic cycle pushed forward from the boundary $\partial\M_{g,2m+n}$. The pullback of \eqref{equality} under the gluing morphism $i_n: \Mbar_{h,k} \times \Mbar_{h,k+n} \rightarrow \Mbar_{g,2m+n}$ gives the equality
		\begin{gather}\label{equality2}
			i_n^*[\Pbar_{g\ra h, 2m, n}]= i_n^*[\P_{g\ra h, 2m, n}]+ i_n^*B_n
		\end{gather}
		in $H^{3h-3+k,3h-3+k}\left(\Mbar_{h,k} \times \Mbar_{h,k+n}\right)$. Now consider the Hodge-K\"unneth decomposition of $H^{3h-3+k,3h-3+k}\left(\Mbar_{h,k} \times \Mbar_{h,k+n}\right)$:
		\begin{gather}
			H^{3h-3+k,3h-3+k}\left(\Mbar_{h,k} \times \Mbar_{h,k+n}\right) \cong \bigoplus_{a+c=b+d=3h-3+k}H^{a,b}(\Mbar_{h,k}) \otimes H^{c,d}(\Mbar_{h,k+n})=\\
			\label{equazione}=H^{3h-3+k,0}(\Mbar_{h,k}) \otimes H^{0,3h-3+k}(\Mbar_{h,k+n}) \oplus ...
		\end{gather}
		Note that by hypothesis we have $H^{3h-3+k,0}(\Mbar_{h,k})\neq 0$, and that this, together with Lemma \ref{lemma1}, implies that the highlighted term in \eqref{equazione} is non-zero. 
		Let us analyze the contributions of the various pieces in \eqref{equality2} belonging to the highlighted term $H^{3h-3+k,0}(\Mbar_{h,k}) \otimes H^{0,3h-3+k}(\Mbar_{h,k+n})$ of the Hodge-K\"unneth decomposition. We start by analyzing $	i_n^*[\Pbar_{g\ra h, 2m, n}]$.
		By Lemma \ref{diagonale}, formula \eqref{eq3}, and excision we have
		\begin{gather}\label{sulbordo}
			i_n^*[\Pbar_{g \rightarrow h,2m, n}]=\alpha [\Delta^n]+B'_n,
		\end{gather}
		where $B'_n$ is an algebraic cycle pushed forward from the boundary $\partial(\M_{h,k} \times \M_{h,k+n})=(\partial\M_{h,k}\times\M_{h,k+n} )\cup (\M_{h,k}\times \partial \M_{h,k+n})$. Let $0\neq e\in H^{3h-3+k,0}(\Mbar_{h,k})$. Then by Lemma \ref{lemma2} we get that $e\otimes \hat{e}\in H^{3h-3+k,0}(\Mbar_{h,k})\otimes  H^{0,3h-3+k}(\Mbar_{h,k})$ is a non-zero contribution to the Hodge-K\"unneth decomposition of $[\Delta]$, and hence that 
		\begin{gather}\label{contribution}
			(id\times\pi_n)^*(e\otimes \hat{e})=e\otimes \pi_n^*(\hat{e})\in H^{3h-3+k,0}(\Mbar_{h,k}) \otimes H^{0,3h-3+k}(\Mbar_{h,k+n})
		\end{gather} 
		is a non-zero contribution to the Hodge-K\"unneth decomposition of the class $ [\Delta^n]\in H^{3h-3+k,3h-3+k}\left(\Mbar_{h,k} \times \Mbar_{h,k+n}\right)$. Now note that \eqref{contribution} restricts non-trivially to the interior. This follows from the fact that, by Lemma \ref{lemma3}, $e\in H^{3h-3+k,0}(\Mbar_{h,k})$ restricts to a non-zero class in $H^{3h-3+k}(\M_{h,k})$ (and the same is true for $\pi_n^*(\hat{e})$). This implies that the contribution \eqref{contribution} cannot be cancelled by $B'_n$ in \eqref{sulbordo}, and thus that $i_n^*[\Pbar_{g \rightarrow h,2m, n}]$ has a non-zero contribution in $H^{3h-3+k,0}(\Mbar_{h,k}) \otimes H^{0,3h-3+k}(\Mbar_{h,k+n})$. On the other hand, we have that the term $i_n^*B_n$ in \eqref{equality2} has no contribution in $H^{3h-3+k,0}(\Mbar_{h,k}) \otimes H^{0,3h-3+k}(\Mbar_{h,k+n})$. The proof of this fact is, with the obvious changes, exactly the same as the one of point $(c)$ in \cite[Theorem C]{AGNES}, hence we omit it here and refer to \cite{AGNES}. Hence \eqref{equality2}, together with the previous considerations, imply that $i_n^*[\P_{g\ra h, 2m, n}]$ has a non-zero contribution in $H^{3h-3+k,0}(\Mbar_{h,k}) \otimes H^{0,3h-3+k}(\Mbar_{h,k+n})$. From this follows that $[\P_{g\ra h, 2m, n}]$ is non-tautological. Indeed, if we suppose by  contradiction that $[\P_{g\ra h, 2m, n}]$   is tautological, then by  \cite[Proposition 1]{GP} (already discussed in Prop. \ref{GrPa}) we have that  $i_n^*[\P_{g\ra h, 2m, n}]$ has tautological K\"unneth decomposition, i.e., it belongs to $RH^*(\Mbar_{h,k}) \otimes RH^*(\Mbar_{h,k+n})$, implying that all its contributions in the Hodge-K\"unneth decomposition are of the form  $H^{p,p}(\Mbar_{h,k}) \otimes H^{q,q}(\Mbar_{h,k+n})$. This gives a contradiction.

	\end{proof}
	As in \cite{AGNES}, using the non-vanishing of certain spaces of holomorphic forms on $\Mbar_{h,k}$, we get:
	\begin{corollary}
		\label{corollario}
		For $g\geq 2$, $m\geq 0$. If $g+m=12$ or $g+m\geq 16$ even, we have that 
		\begin{gather*}
			[\P_{g\ra 1, 2m, n}]\in H^{2g+2m-2} (\M_{g,2m+n})
		\end{gather*}
		is non-tautological for all $n\geq 0$. For $g\geq 4$ and $m\geq 0$, if $g+m\geq 17$ odd, then
		\begin{gather*}
			[\P_{g\ra 2, 2m, n}]\in H^{2g+2m} (\M_{g,2m+n})
		\end{gather*}
		is non-tautological for all $n\geq 0$.
		\begin{proof}
			The first statement is obtained by applying Theorem \ref{thm:main} for $h=1$ and $k\geq 11$ odd, $k\neq 13$. Indeed by Proposition \ref{nonvanishing}, we have for these $k$'s that $H^{k,0}(\Mbar_{1,k})\neq 0$. For the second statement we use that, for $h=2$ and $k\geq 14$ even, we have $H^{k+3,0}(\Mbar_{2,k})\neq 0$.	
		\end{proof}
	\end{corollary}
	This proves Theorem \ref{teoprincipale}.
		There exists a non-tautological algebraic cycle on $\M_{g,n}$ for any $n\geq 0$ such either $2\leq g\leq 12$ and $2g+n\geq 24$, or $2g+n\geq 32$. In particular, for such $g$ and $n$ we have $RH^{2*}(\M_{g,n})\neq H^{2*}(\M_{g,n})$.
  
	\begin{remark}
     We observe that to deduce Theorem \ref{teoprincipale} it is sufficient to consider the first statement in Corollary \ref{corollario} together with the results in \cite{AGNES} (see the introduction).
	\end{remark}

\end{document}